\documentclass[preprint,11pt,2015]{elsarticle}
\journal{Discrete Mathematics}
\usepackage{amsmath,amssymb,latexsym,color,epsfig,a4, enumerate}
\usepackage{a4wide}

\newtheorem{theorem}{Theorem}[section]
\newtheorem{lemma}[theorem]{Lemma}
\newtheorem{claim}[theorem]{Claim}
\newtheorem{proposition}[theorem]{Proposition}
\newtheorem{observation}[theorem]{Observation}

\newtheorem{problem}[theorem]{Problem}

\newtheorem{definition}[theorem]{Definition}

\def\cf{{\cal F}}
\def\cP{{\cal P}}

\def\dang{{\rm dang}}
\def\depth{{\rm depth}}

\newcommand{\gnp}{\ensuremath{\mathcal{G}_{n,p}}}

\newenvironment{proof}{\noindent{\bf Proof\,}}{\hfill$\Box$}

\begin{document}

\begin{frontmatter}

\title{Creating cycles in Walker-Breaker games}

\author[add1]{Dennis Clemens}
\address[add1]{Technische Universit\"at Hamburg-Harburg, Institut f\"ur Mathematik, Am Schwarzenberg-Campus 3, 21073 Hamburg, Germany}
\ead{dennis.clemens@tuhh.de}
\author[add2]{Tuan Tran}
\address[add2]{Freie Universit\"at Berlin, Institut f\"ur Mathematik, Arnimallee 3, 14195 Berlin, Germany}
\ead{tuan@math.fu-berlin.de}

\begin{abstract}
We consider biased $(1:b)$ Walker-Breaker games: Walker and Breaker alternately claim edges
of the complete graph $K_n$, Walker taking one edge and Breaker claiming $b$ edges
in each round, with the constraint that Walker needs to choose
her edges according to a walk. As questioned in a paper by
Espig, Frieze, Krivelevich and Pegden, we study how long a cycle Walker is able to create
and for which biases $b$ Walker has a chance to create a cycle of given constant length.
\end{abstract}

	
\end{frontmatter}

\section{Introduction}

In this paper we consider a variant of the well-known Maker-Breaker games, 
which have  been studied by various researchers (see e.g.~\cite{BeckBook,HKSSBook}). Biased $(a:b)$
Maker-Breaker games are played as follows: Given a hypergraph $(X,\cf)$,
the players alternately take turns in claiming elements from
the {\em board} $X$, where Maker claims $a$
elements 
in each round followed by Breaker claiming $b$ elements. 
If Maker
manages to occupy all the elements of one of the {\em winning sets}
$F\in \cf$, she is declared to be the winner of the game.
Otherwise, if Breaker claims at least one element in each of the winning sets,
Breaker wins the game. The values $a$ and $b$ are called
the {\em biases} of Maker and Breaker, respectively. If $a=b=1$ holds, the game
is referred to as an {\em unbiased} game.

Throughout the paper, we will concentrate on the case when $X=E(K_n)$, i.e., the two players
claim edges of the complete graph on $n$ vertices. 
Many natural unbiased Maker-Breaker games on $K_n$
significantly favor Maker.
Indeed, Maker can create a connected spanning subgraph of $K_n$
within $n-1$ rounds. 
For even $n$, she is able to create a perfect matching within $\frac{n}{2}+1$
rounds, as shown by Hefetz, Krivelevich, Stojakovi\'c and Szab\'o~\cite{HKSS2009b}.
She can create a Hamilton cycle within $n+1$ rounds~\cite{HS2009},
and for $k\geq 2$ she can create a $k$-connected spanning subgraph
within $\lfloor \frac{kn}{2} \rfloor +1$ rounds~\cite{FH2014}. So, for any of these structures
the number of rounds she needs to play is at most one larger than the minimal size of a winning set.

Due to this overwhelming power of Maker in these kinds of games,
it is natural to consider variations that help to increase Breaker's power. One natural way is to increase Breaker's bias
as initiated by Chv\'atal and Erd\H{o}s~\cite{CE1978}, and continued further in e.g.~\cite{BL2000, GS2009, K2011, NSS2014}.
A second possibility for increasing Breaker's power is to decrease the number of winning sets,
by making the board sparser. Games on sparser graphs are discussed in e.g.~\cite{BFHK2012, CFKL2012, FGKN2012, HKSS2009a, SS2005}.

Finally, a third option is to restrict Maker's choices according to some pre-defined rule.
In this regard, we consider $(1:b)$ Walker-Breaker games on $K_n$, which are played like
Maker-Breaker games, just with one additional constraint: Walker (playing the role of Maker)
has to choose her edges according to a walk. 
These games have been introduced recently by Espig, Frieze, Krivelevich and Pegden~\cite{EFKP2014},
and their precise rules are as follows. Walker and Breaker alternately occupy edges of $K_n$, Walker choosing one edge and
Breaker choosing $b$ edges in each round. At any moment,
we identify some vertex $v$ as the {\em position} of Walker in the game.
For her next move, Walker then needs to choose an edge $vw$ incident with $v$
which has not been chosen by Breaker so far (but could have been chosen by Walker earlier).
Walker claims this particular edge (in case she did not choose and claim it before),
and makes $w$ her new position. In contrast, there are no restrictions for Breaker.
By the end of the game, Walker wins if she occupies all the elements of one
winning set, while Breaker wins otherwise.

It is easy to see that in unbiased Walker-Breaker games, Walker cannot hope to create
a spanning structure, contrary to Maker-Breaker games where Maker can easily do so even 
when Breaker's bias is of size $(1-o(1))\frac{n}{\ln n}$, see e.g.~\cite{GS2009,K2011}.
Indeed, Breaker can easily isolate a vertex from Walker's graph
by fixing a vertex right after Walker's first move and then
always claiming the edge between this particular vertex and
Walker's current position.
So, Walker will not be able to create a spanning tree or a Hamilton cycle. 
It thus becomes natural to ask how large a structure 
Walker is able to create.

Espig, Frieze, Krivelevich and Pegden~\cite{EFKP2014} studied how many vertices
Walker can visit for various variants of the game, i.e.~they studied how large a tree 
Walker can obtain. For instance, when $1\leq b= O(1)$, they showed that $n-2b+1$ is the 
largest number of vertices that Walker can visit.  Moreover, in the case when Walker is not allowed to return to vertices, 
they proved this number to be $n-\Theta(\ln n)$ for $b>1$, with the value being precisely $n-2$ for $b=1$ and $n\geq 6$.
Adressing a question of Espig, Frieze, Krivelevich and Pegden~\cite{EFKP2014}, we now discuss games in which Walker aims to
occupy large cycles. In particular, we prove the following theorems.

\begin{theorem}\label{Walker_cycle1}
	For large enough $n$, Walker can create a cycle of length $n-2$
	in the unbiased Walker-Breaker game on $K_n$, while Breaker (as first player)
	can prevent any longer cycle.
\end{theorem}

\begin{theorem} \label{Walker_cycle2}
Let $n$ and $b=b(n)\leq \frac{n}{\ln^2 n}$ be integers. Then, in the $(1:b)$ Walker-Breaker game on $K_n$, Walker can create a cycle of length $n-O(b)$.
\end{theorem}

Since Breaker can easily isolate $b$ vertices from Walker's graph, the result above is clearly best possible up to a constant factor in the non-leading term.
Finally, we consider the Walker-Breaker $C_k$-game in which Walker's goal is to create a cycle of length $k$.
We show that the largest bias $b$ for which Walker has a winning strategy in this game 
is of the same order as in the corresponding Maker-Breaker game, see \cite{BL2000}. That is, we prove the following theorem.

\begin{theorem}\label{H-game}
Let $k\geq 3$. Then the largest bias $b$ for which Walker wins the $(1:b)$ Walker-Breaker $C_k$-game
is of the order $\Theta(n^{\frac{k-2}{k-1}})$.
\end{theorem}

The rest of the paper is organized as follows. In Section~\ref{WB:preliminaries},
we collect some useful theorems connected to local resilience and random graphs.
In Section~\ref{WB:diameter} we show that Walker can create almost spanning graphs of small diameter
within a small number of rounds. Then, applying all these results, we prove
Theorems~\ref{Walker_cycle1} -- \ref{H-game} in Section~\ref{WB:cycle}. In Section~\ref{WB:remarks}, we conclude with some remarks and open problems.

\subsection{Notation}
Our graph theoretic notation is rather standard and follows \cite{West}. The vertex and edge sets of a graph $G$ are denoted by $V(G)$ and $E(G)$.
We write $v(G)$ for the number of vertices of $G$, and $e(G)$ for the number of edges. For $x\in V(G)$ and $A \subseteq V(G)$, $e_G(x,A)$ stands for the number of vertices in $A$ adjacent to $x$. When $A=V(G)$ we abbreviate $d_G(v)$ for $e_G(x,V(G))$. The {\em minimum degree} of $G$ is defined to be $\delta(G):=\min_{v\in V(G)} d_G(v)$. Given two sets $A, B \subseteq V(G)$,
not necessarily disjoint, we use $e_G(A,B)$ to denote the number of edges with one endpoint 
in $A$ and the other in $B$. We write
$G[A]$ for the subgraph of $G$ induced by a set $A \subseteq V(G)$. 

The {\em depth} of a rooted tree $T$, denoted by $\hbox{depth}(T)$, is the length of the longest path from the root to a leaf. 

Let $n$ be a positive integer and let $0 < p < 1$. The Erd\H{o}s-R\'enyi model $\mathcal{G}_{n,p}$ is a random subgraph $G$ of $K_n$,
obtained by retaining each edge of $K_n$ in $G$
independently at random with probability $p$. Let $\mathcal{P}$ be a graph property, and consider a sequence of probabilities $\{p(n)\}_{n=1}^{\infty}$. We say
that $\mathcal{G}_{n, p(n)}\in \mathcal{P}$ asymptotically almost surely, or a.a.s. for brevity, if the probability that $\mathcal{G}_{n, p(n)}\in \mathcal{P}$ tends to $1$ as $n$ goes to infinity.

To simplify the presentation, we often omit floor and ceiling signs whenever they are not essential. 
We also assume that the parameter $n$ (which always denotes the number of vertices of the host
graph) tends to infinity and therefore is sufficiently large whenever necessary. All our
asymptotic notation symbols $(O, o, \Omega, \omega, \Theta)$ are relative to this variable $n$.

Throughout the paper, the graph induced by Breaker's edges is denoted by $B$ and the graph induced by Walker's edges is denoted by $W$. We also write $F$ for the graph induced by edges which have not been claimed by either Breaker or Walker.

\section{Preliminaries}\label{WB:preliminaries}

In order to prove results for biased games, we will adapt a method that was motivated by
Bednarska and \L uczak~\cite{BL2000}, and recently extended by
Ferber, Krivelevich and Naves~\cite{FKN2013}. Providing Walker with a strategy
which mixes a randomized strategy with a deterministic one, we ensure that Walker can
generate a random graph throughout the game while ensuring that most
of its edges will be claimed by herself.  For that reason, we make use of the following definition, similarly to \cite{FKN2013}.

\begin{definition}
For $n\in\mathbb{N}$, let $\cP=\cP(n)$ be some graph property that is monotone increasing, and let $0\leq \varepsilon,p=p(n)\leq 1$. Then 
$\cP$ is said to be $(p,\varepsilon)$-resilient if a random graph $G\sim \gnp$ a.a.s. has the following property:
For every $R\subseteq G$ with $d_R(v)\leq \varepsilon d_G(v)$ for every $v\in V(G)$ it holds that $G\setminus R\in \cP$.
\end{definition}

Using this definition, the following theorems are known to be true.

\begin{theorem}[Theorem 1.1 in \cite{LS2012}]\label{resilience_new}
For every $\varepsilon >0$, there exists a constant $C=C(\varepsilon)$ such that for $p\geq \frac{C\ln n}{n}$ 
the property $\cP$ of containing a Hamilton cycle is $(p,\tfrac{1}{2} - \varepsilon)$-resilient.
\end{theorem}

\begin{theorem}[Corollary of Theorem 15 in \cite{NSS2014}]\label{resilience_H}
Let $k\ge 3$. Then there exist constants $C,\gamma>0$ such that for  
$p\ge Cn^{-\frac{k-1}{k-2}}$ the property $\cP$ of containing a copy of 
$C_k$ is $(p,\gamma)$-resilient.
\end{theorem}

We also make use of the next theorem which roughly states the following: if 
after playing on $K_n$ for some rounds Walker has a graph which connects every pair of vertices
by short paths, while most edges are still free, then Walker can continue the game in such a way that
she obtains a graph with a pre-defined $(p,\varepsilon')$-resilient
property even if Breaker plays with a bias of order $\Theta(\frac{\varepsilon'}{p})$. 
As we might apply this result later to a smaller part of the given board (where this smaller part is identified with $K_n$),
we allow the mentioned paths to use edges not belonging to $K_n$. For the proof of Theorem \ref{random} this does not make a difference, as we will only need the property that these paths help Walker to reach any vertex within a small number of rounds.

\begin{theorem} 
\label{random}
For every positive $\varepsilon < 10^{-5}$ and every large enough $n\in\mathbb{N}$ the following holds.
Let $\frac{\ln n}{\varepsilon n} \le p=p(n)\leq 1$, let 
$d \in \mathbb{N}$ and let $\cP=\cP(n)$
be a monotone $(p,4\varepsilon)$-resilient graph property. Assume a $(1:\frac{\varepsilon}{30(d+1)p})$
Walker-Breaker game on $K_n$ is in progress, where
the graph $F$ of free edges satisfies $\delta(F)\geq (1-\varepsilon)n$,
and where Walker's current graph $W_0$ has the property that between every two vertices in $V(K_n)$ there is a path of length at most $d$.
Then, Walker has a strategy for continuing the game that creates a graph $W'\in \cP$.
\end{theorem}

Since this theorem is motivated by and follows very closely the argument for Theorem 1.5 in \cite{FKN2013}, its proof is postponed to Section \ref{WB:appendix}.



\section{Creating almost spanning graphs with small diameter}\label{WB:diameter}

Besides the previous results on $(p,\varepsilon)$-resilient properties, our results will also make use of Walker's ability to create a graph of small diameter covering almost the whole vertex set $V(K_n)=[n]$, within a small number of rounds. 

\begin{proposition}\label{diameter}
	For every large enough integer $n$ the following holds.
	Let $b \le \frac{n}{\ln^2 (n)}$ be a positive integer, and let
	$r=\frac{\ln n}{\ln \left(\frac{n}{200b}\right)}$. Then, in the biased $(1:b)$ Walker-Breaker game on $K_n$, Walker has
	a strategy to create a graph on $n-400b$ vertices, with diameter at most $2\lfloor r\rfloor+6$, within at most $6n$ rounds.
\end{proposition}

{\bf Proof of Proposition \ref{diameter}}
Whenever necessary, let us assume that $n$ is large enough.
By assumption it follows that 
$r=O\left(\frac{\ln n}{\ln \ln n}\right).$
For simplicity of notation we set $c_1=\frac{1}{2\lfloor r \rfloor+2}\leq \frac{1}{4}$
and $c_2=2\lfloor r \rfloor+2$.

The main idea is to create a tree in some kind of breadth first search manner, by iteratively
attaching stars to the vertices of a smaller tree.
At a given moment during the game, assume that $T$ is the tree which Walker has created so far and that Walker's current position is $v\in V(T)$. Then, by {\em attaching a star} of size $s\in \mathbb{N}$ to $v$ we mean a strategy of creating a star $S$ with center $v$ in the following way: as 
long as the star $S$ has not reached  
size $s$, Walker proceeds to an untouched vertex $w$ if her current position is $v$ (thus enlarging $S$ by one edge), or she proceeds to the vertex $v$ if her current position is a leaf of $S$.
Notice that Walker can easily follow this strategy 
if there are at least $(2b+1)s$ vertices $w\in V\setminus V(T)$ such that the edge $vw$ is free.

In the following we describe a strategy for Walker
in the $(1:b)$ game on $K_n$. Afterwards, we show that she can follow that strategy,
and while doing so, she creates a tree as required.
Initially set $L_0=\{v_0\}$, where $v_0$ is Walker's start vertex. We consider $v_0$ to be 
the root of Walker's tree $T$.
Initially, $V(T)=L_0=\{v_0\}$.
Now, Walker plays according to three stages, where 
in the first stage she progressively enlarges her tree
by attaching large stars to the leaves of her current tree $T$, while in the other two stages we also take non-leaves into account.

{\bf Stage I.} Walker starts by creating a tree
$T$ of size $c_1n$ and depth at most $\frac{c_2}{2}$.
Stage I is split into several phases.

{\em Phase 1.} Walker attaches a star of size $\frac{n}{200b}$ to the root $v_0$. Once she is done with the star, she proceeds to Phase~2.

{\em Phase $j>1$.} Let $L_{j-1}$ be the set of leaves of Walker's tree immediately at the end of Phase~$j-1$.
In Phase $j$, Walker enlarges her tree by attaching
stars of size $\frac{n}{100b}$ to at most $\frac{|L_{j-1}|}{2}$ vertices of $L_{j-1}$, but only as long as
$v(T)\leq c_1n$. She proceeds as follows:

If she already made $\frac{|L_{j-1}|}{2}$ star attachments
in this phase, then Walker proceeds with Phase $j+1$.
Otherwise, assume that Walker already finished $i\leq \frac{|L_{j-1}|}{2} -1$ stars. Then she claims her $(i+1)^{\text{st}}$ star as follows:
She identifies a vertex $v\in L_{j-1}$
which is not the center of one of her stars yet and which has the smallest possible Breaker-degree among all such vertices. Within at most $2j$ rounds 
she walks to this vertex $v$ by using the edges of her current tree $T$. If  
$v(T)\leq c_1n-\frac{n}{100b}$, Walker then attaches
a star of size $\frac{n}{100b}$ to $v$, and she repeats Phase $j$. Otherwise, if 
$c_1n-\frac{n}{100b}<v(T)<c_1n$, 
Walker attaches a star of size 
$c_1n-v(T)$, and proceeds to Stage II.
(We will see later that the number of phases is at most $\frac{c_2}{2}$.)

{\bf Stage II.} From now on, let $T_1$ denote the tree which Walker has constructed 
by the end of Stage I,
with $v(T_1)=c_1n$.
Throughout Stage II, Walker maintains a tree
$T$ of depth at most $\frac{c_2}{2}+1$,
by attaching large stars to 
the vertices of $T_1$. Assume Walker already attached $i-1$ such stars, and
now she plays according to Stage II for the $i^{\text{th}}$ time.
Let $V_i$ be the set of vertices that are not contained in her tree so far.
If $|V_i|\leq 50c_2b$, then Walker proceeds to Stage III.
Otherwise, she fixes an arbitrary vertex
$z_i\in V(T_1)$
with $d_F(z_i,V_i)>0$ (the existence will be proven later), which minimizes $d_B(z_i,V_i)$.
Then, in the following at most $c_2+2$ rounds, Walker walks to $z_i$ using the edges of her current tree.
Afterwards, as long as possible, Walker creates a star with center $z_i$
and leaves in $V_i$. That is, as long as possible, she claims a free edge between $z_i$ and $V_i$
in every second round (by alternately walking between $z_i$ and distinct vertices from $V_i$).
When this is no longer possible, she repeats Stage II.

{\bf Stage III.} From now on, let $T_2$ denote the tree which Walker has constructed 
by the end of Stage~II, with $v(T_2)\geq n-50c_2b$.
At the beginning of Stage III, Walker walks to a vertex 
$z_0\in V(T_2)$ with $d_F(z_0) \geq \frac{10n}{11}$, within at most $c_2+2$ rounds. 
Then, throughout Stage III, Walker maintains a tree
$T$ of depth at most $\frac{c_2}{2}+2$, by
attaching large stars to the vertices of $V(T_2)$. 
Assume that Walker already attached $i-1$ such stars, and her current position is $z_{i-1}$, for some $i\geq 1$.
Let $W_i$ be the set of vertices that have not been visited by Walker so far.
If $|W_i|\leq 400b$, then Walker stops playing.
Otherwise, she continues to attach stars to her current tree.
For this, she identifies a vertex 
$z_i\in V(T_2)$ such that (i) $z_{i-1}z_i$ is a free edge,
(ii) $d_F(z_i)\geq \frac{9n}{10}$ and (iii) $d_B(z_i) \leq 100b$. She immediately walks to $z_i$ using the edge $z_{i-1}z_i$. Afterwards, Walker creates a star of size $\frac{|W_i|-d_B(z_i,W_i)}{2b+1}-1$ with center $z_i$ and with leaves in $W_i$. She then repeats Stage III.

Obviously, if Walker can follow the strategy until $|W_i|\leq 400b$ holds at some point in Stage~III
and such that Stage I consists of at most $\frac{c_2}{2}$ phases,
then she creates a tree as required. Indeed, under these assumptions her final tree
needs to have at least $n-400b$ vertices, and diameter at most $2(\frac{c_2}{2}+2)=2\lfloor r\rfloor + 6$, as the latter increases by at most 2 in every phase of Stage I, in Stage II and in Stage III.
Thus, it remains to prove that she can follow the proposed strategy, and that 
the above assumptions are satisfied at some point in the game, but after at most $6n$ rounds. We start with the following useful claim.

\begin{claim}\label{cl:rounds}
Assume that Walker can follow her strategy and let $T$ be her tree at any moment during the game. Let 
$d_{\geq 2}(T)$ denote the number of vertices $v$ such that $d_T(v)\geq 2$. Then $2e(T)+2 \depth(T)\cdot d_{\geq 2}(T)$ is an upper bound on the number of rounds played so far.
\end{claim}

\begin{proof}
Walker's strategy consists of two different actions. On one hand, she makes star attachments, walking twice along each edge added to the tree. 
On the other hand, after finishing a star, she moves to a vertex, which becomes the next center for a star attachment. The latter is done either by traversing
edges of $T$ (in Stage I and II), which needs at most $2\cdot\text{depth(T)}$ rounds, or by doing just one move (in Stage III). As $d_{\geq 2}$ is precisely the number of star attachments made so far, the claim follows.
\end{proof}

{\bf Stage I.} We now show that Walker can follow the strategy of Stage I until her tree has size $c_1n$, while the depth is not larger than $\frac{c_2}{2}=\lfloor r\rfloor +1$. To do so, let $T_j$ denote Walker's tree at the end of Phase $j$ (and observe that $T_j$ has depth $j$), let 
$L_j$ denote the set of leaves of $T_j$ (as stated in the strategy description), and let $R_j$ denote the number of rounds until the end of Phase $j$ (including all previous phases). The following claim provides us with some useful inequalities for the analysis of Walker's strategy.

\begin{claim} 
	\label{diameter_first-claim_}
	Let $n$ be large enough.
	As long as Walker can follow the strategy of Stage I, the following holds for every Phase $j$:
	\begin{enumerate}
		\item $|L_0|=1$ and $|L_j| \leq \frac{n}{200b} |L_{j-1}|$. Moreover, $|L_j| = \left(\frac{n}{200b}\right)^j$ if $v(T_j)< c_1n$.
		\item $R_j\leq 2e(T_j)+2j\cdot v(T_{j-1})<5 \cdot \left(\frac{n}{200b}\right)^j$.
	\end{enumerate}
\end{claim}

\begin{proof}
	If Walker can follow the strategy, then in 
	Phase 1 she creates a star of size $\frac{n}{200b}$, ensuring that $|L_1|=\frac{n}{200b}$.
	Moreover, in Phase~$j$
	Walker attaches at most $\frac{|L_{j-1}|}{2}$ stars of size $\frac{n}{100b}$ to the vertices of $L_{j-1}$, giving
	${|L_j|\leq \frac{n}{200b}|L_{j-1}|}$. If $v(T_j)<c_1n$, she attaches exactly $\frac{|L_{j-1}|}{2}$ stars of size $\frac{n}{100b}$ to the vertices of $L_{j-1}$, giving ${|L_j|= \frac{n}{200b}|L_{j-1}|}$, which proves the first part of the claim.
	The second part of the claim is obvious if $j=1$.	So, let $j>1$. Since $\left(\frac{n}{200b}\right)^{j-1} =|L_{j-1}|<n$ holds when $v(T_{j-1})<c_1n$, we conclude 
	that $j \leq \lfloor r \rfloor+1=o(\ln n)$ for every Phase $j$ in Stage I. Moreover,
	$v(T_{j-1})=\sum_{i=0}^{j-1} |L_i|<2\cdot \left(\frac{n}{200b}\right)^{j-1}$ and $e(T_j)<v(T_j)<2\cdot \left(\frac{n}{200b}\right)^j$. Thus, applying Claim \ref{cl:rounds}, we conclude
	$$R_j\leq 2e(T_j)+2j\cdot v(T_{j-1})\leq 4\cdot\left(\frac{n}{200b}\right)^j+4j\cdot \left(\frac{n}{200b}\right)^{j-1}<5\cdot \left(\frac{n}{200b}\right)^j,$$
	where in the last inequality we used the fact that $\frac{n}{200b}=\Omega(\ln n)$ and $j=o(\ln n)$.
\end{proof}

\begin{claim}\label{BdegreeLeave}
	Let $n$ be large enough.
	Immediately before Walker starts building a star with center vertex $v$ in Stage I,
	we have $d_B(v)\leq 0.2n$.
\end{claim}

\begin{proof}
	When Walker starts her first star at $v_0$, there are no Breaker edges at all.
	Therefore, for $v=v_0$ the statement is obvious, and we can consider Phase $j$, for some $j>1$.
	Let $v\in L_{j-1}$ be a vertex at which Walker wants to attach a star,
	according to the proposed strategy. 
	Then, before starting the star, $v$ belongs to a set of at least $\frac{|L_{j-1}|}{2}$ vertices
	of $L_{j-1}$ that still have degree~1 in Walker's graph. Since we played at most $R_j$ rounds so far,
	Breaker has claimed at most $b\cdot R_j$ edges, which implies that the average Breaker-degree
	of all these at least $\frac{|L_{j-1}|}{2}$ vertices, is bounded from above by
	\begin{align*}
		\frac{2b\cdot R_j}{|L_{j-1}|/2} < 20b \cdot \left(\frac{n}{200b}\right)^j\big/\left(\frac{n}{200b}\right)^{j-1}=0.1n,
	\end{align*}
	where the first inequality follows from Claim \ref{diameter_first-claim_}.
	Since Walker, by following the strategy, chooses the vertex $v$ such that its Breaker-degree is minimal,
	we obtain $d_B(v)\leq 0.1n$ at the moment when Walker fixes $v$ for attaching a star.
	She may walk to $v$ within in the following $2j$ rounds, but even afterwards
	$d_B(v)\leq 0.1n+2jb<0.2n$ holds.
\end{proof}

With the previous claims in hand, we can conclude that Walker is able to follow the strategy of Stage I.
Indeed, Walker can easily follow Phase 1.
Afterwards, whenever Walker aims to attach a new star
to some vertex $v$, then the number of vertices $w\in V\setminus V(T)$
with $vw$ being free is at least
$$|V\setminus V(T)|-d_B(v)\geq (1-c_1)n-0.2n>(2b+1)\cdot\frac{n}{100b},$$
where the first inequality holds by Claim \ref{BdegreeLeave} and since $v(T)\leq c_1n$, as long as Walker plays according to Stage I. Thus, there are enough free edges to make sure that Walker can attach a star of size $\frac{n}{100b}$ as required. 

Moreover, when she finishes her tree of size $c_1n$ during Stage $j$, then $\left(\frac{n}{200b}\right)^{j-1}=|L_{j-1}|<n$ holds. Hence, her final tree has depth $j \leq \lfloor r\rfloor+1 = \frac{c_2}{2}$. As Walker makes at most $200b$ star attachments throughout Stage I (as each such attachment adds $\frac{n}{200b}$ vertices to the tree), we further know, by Claim \ref{cl:rounds}, that Stage I lasts at most $2c_1n+200b \cdot 2(\lfloor r\rfloor + 1)\leq \frac{n}{2} + o(n)<n$ rounds.

{\bf Stage II.} To proceed with the discussion of Stage II,
we first prove the following useful claim.

\begin{claim}\label{StageIIdegree}
	Let $n$ be large enough.
	Assume that Walker can follow the strategy of Stage~II.
	Then, as long as $i \leq \frac{c_1n}{2}$,
	we have that  
	$d_B(z_i,V_i)\leq 20c_2b$
	at the moment when Walker starts to attach a star
	at $z_i$.
\end{claim}

\begin{proof}
	As long as $i \leq \frac{c_1n}{2}$ star attachments happened in Stage II, less than $3n$ rounds were played, which follows by Claim \ref{cl:rounds}, since $e(T)\leq n$ and $\text{depth}(T)\leq \frac{c_2}{2}+1$, and since the number of all previous star attachments is bounded by $200b+i$.
	It follows, when Walker identifies $z_i$, that we have $e(B)\leq 3bn$ and
	$|V(T_1)\setminus \{z_1,\ldots,z_{i-1}\}|\geq \frac{c_1n}{2}.$
	Thus, at this moment,
	the Breaker-degree of $z_i$ is at most
	$$\frac{2e(B)}{|V(T_1)\setminus \{z_1,\ldots,z_{i-1}\}|}\leq 12c_2b,$$
	as Walker chooses $z_i$ with minimal Breaker-degree.
	After Walker identified $z_i$ for her $i^{\text{th}}$
	repetition of Stage II, she needs at most $c_2+2$ rounds to make $z_i$ her new position.
	So, when she starts creating the mentioned star,
	$d_B(z_i,V_i)\leq 12c_2b+(c_2+2)b< 20c_2b$ holds.
\end{proof}

By this claim we know that Walker can follow Stage II
as long as $i\leq \frac{c_1n}{2}$ and $|V_i|> 50c_2b$.
It thus remains to show that $|V_i|\leq 50c_2b$
holds (and Walker proceeds to Stage III) after at most $\frac{c_1n}{2}$ star attachments in Stage II.

For this, consider for $i\leq \frac{c_1n}{2}$ the $i^{\text{th}}$ star attachment in Stage II, where we may assume that $|V_i|>50c_2b$ still holds.
As, during this star attachment,
Walker claims one edge within every two rounds, while Breaker claims $2b$ edges,
Walker creates a star with center $z_i$ of size at least
$\frac{|V_i|-d_B(z_i,V_i)}{2b+1}-1$.
Using Claim \ref{StageIIdegree} this yields
\begin{align*}
	|V_{i+1}| \leq |V_i|- \Big(\frac{|V_i|-d_B(z_i,V_i)}{2b+1}-1 \Big)
	\leq \frac{2b}{2b+1}|V_i|+\frac{20c_2b}{2b+1}+1
	\leq \left(1-\tfrac{1}{4b+2}\right)|V_i|.
\end{align*}
So, we conclude
$50c_2b < |V_{i}|\leq \left(1-\tfrac{1}{4b+2}\right)|V_{i-1}| \leq \ldots \leq \left(1-\tfrac{1}{4b+2}\right)^{i-1}n<e^{-(i-1)/(4b+2)}n.$
But this does not hold for $i>\frac{c_1n}{2}$ as $\frac{c_1n}{b}=\omega(\ln n)$. Thus, 
after at most $\frac{c_1n}{2}$ rounds, Walker needs to reach a point when $|V_i|\leq 50c_2b$.

{\bf Stage III.} When Walker enters Stage III, we have $v(T_2)\geq n - 50c_2b$,
while at most $3n$ rounds were played so far.
Moreover, the following claim shows that Walker can follow her strategy for at least $\frac{n}{4}$ star attachments, as long as $|W_i|\geq 400b$.

\begin{claim}\label{StageIIIidentification}
	Let $n$ be large enough.
	As long as $i \leq \frac{n}{4}$ and $|W_i| \geq 400b$,
	Walker can always identify a vertex $z_i$, as described by her strategy in Stage III. 
\end{claim}

\begin{proof}
	At the beginning of Stage III, less than $3n$ rounds were played.
	The number of vertices $v\in V(K_n)$ such that $d_F(v)<\frac{10n}{11}-1$ is at most $$\frac{2(e(B)+e(W))}{n/11}\leq \frac{2(3bn+3n)}{n/11}=o(n)=o(v(T_2)).$$ Thus, there exists a vertex $z_0\in V(T_2)$ such that $d_F(z_0)\geq \frac{10n}{11}$. Walker then needs at most $c_2+2$ rounds to reach $z_0$ from the position which she has at the end of Stage II. Hence, at the time when Walker visits $z_0$, we must have $d_F(z_0)\geq \frac{10n}{11}-(c_2+2)(b+1)>\frac{9n}{10}$.
	
	Let us consider now the remaining vertices $z_i$ with $i \leq \frac{n}{4}$.
	As long as $i \leq \frac{n}{4}$, we see that at most $3n+(c_2+2)+i+2n <6n$ rounds were played.
	Indeed, until the end of Stage II the game lasts at most $3n$ rounds, we may need $c_2+2$ rounds to reach the vertex $z_0$ at the beginning of Stage III,
	for each new identification of a vertex $z_j$ Walker chooses the edge $z_{j-1}z_j$, and the star attachments last at most $2n$ rounds in total.
	Thus, right before Walker's move from $z_{i-1}$ to $z_i$, we must have $e(B) \leq 6bn$ and $e(W)\leq 6n$. Now, right at this moment consider 
	$$X:=\big\{v\in V(T_2)\setminus \{z_1,\ldots,z_{i-1}\}: \ d_F(v)\geq \frac{9n}{10}\big\}.$$ 
	The number of vertices $v\in V(K_n)$ with $d_F(v)\leq \frac{9n}{10}-1$ is at most $$\frac{2(e(B)+e(W))}{n/10}\leq \frac{2(6bn+6n)}{n/10}=o(n).$$ Thus, $|X|\geq |V(T_2)\setminus \{z_1,\ldots,z_{i-1}\}|-o(n) > \frac{n}{2}$. 
	Moreover, by the choice of $z_{i-1}$ and the size of the star attached to $z_{i-1}$, we obtain
	$$d_F(z_{i-1}) \geq \frac{9n}{10}-(2b+1)\cdot \frac{|W_{i-1}|-d_B(z_{i-1},W_{i-1})}{2b+1} \geq \frac{9n}{10}-|W_{i-1}|> \frac{8n}{9},$$
	as $|W_i|\leq 50c_2b$.
	Hence $d_F(z_{i-1},X) > \frac{n}{3}$.
	Furthermore, since $e(B)<6bn$, at most $\frac{n}{8}$ vertices in $X$ have Breaker-degree larger than $100b$. It follows that there is a vertex $z_i\in X$ such that $z_{i-1}z_i$ is a free edge and $d_B(z_i) \leq 100b$.
\end{proof}

We finally show that Walker can follow Stage III and create a tree as required. As long as $i\leq \frac{n}{4}$ and $|W_i|\geq 400b$ she can identify a vertex $z_i$ as claimed. 
She then obviously can attach a star as described, and while doing so, she ensures that
\begin{align*}
	|W_{i+1}| \leq |W_i|- \Big(\frac{|W_i|-d_B(z_i,W_i)}{2b+1}-1 \Big)
	\leq \frac{2b}{2b+1}|W_i|+51
	\leq \left(1-\frac{1}{4b+2}\right)|W_i|.
\end{align*}
In particular, as long as $|W_i|\geq 400b$, this yields
$$400b\leq |W_{i}|\leq \left(1-\tfrac{1}{4b+2}\right)|W_{i-1}|\leq \ldots \leq \left(1-\tfrac{1}{4b+2}\right)^{i-1} \cdot 50c_2b <e^{-(i-1)/(4b+2)}\cdot 50c_2b.$$
However, as $\frac{n}{b}=\Omega(\ln^2 n)$, this cannot hold for $i > \frac{n}{4}$. That is, after at most $\frac{n}{4}$ star attachments (and therefore at most $6n$ rounds)
Walker must have a tree on $n-400b$ vertices, of depth
at most $\frac{c_2}{2}+2$, and diameter at most
$c_2+4=2\lfloor r \rfloor + 6$.
\hfill $\Box$



\section{Main proofs}\label{WB:cycle}

\subsection{The unbiased game}

In the following we prove {\bf Theorem \ref{Walker_cycle1}}.

We start with {\bf Breaker's part}. Assuming he plays as first player,
his strategy is as follows. At the beginning of the game, he fixes a vertex $w_1$
which is not the start vertex $v_0$ of Walker. As long as Walker has a component of size
less than $n-2$, Breaker's strategy is to claim the edge between $w_1$ and Walker's current position.
In case this edge is not free, he claims another arbitrary edge. Note that this in particular means
that Breaker's first edge is $v_0w_1$, and inductively Walker has no chance to make $w_1$ to
become her next position, as long as her graph is a component of size smaller than $n-2$.

If Walker does not manage to create a component of size $n-2$, then there will not be a cycle
of length $n-2$, and we are done. So, we can assume that there is a point in the game where
Walker's component $K$ reaches size $n-2$, with $v'$ being the last vertex added to it.
Let $w_2$ be the other vertex besides $w_1$ which does not belong to $K$, and note that the only free edges
incident with $w_1$ are $w_1w_2$ and $w_1v'$. From now on, Breaker always
claims the edge between $w_2$ and the current position of Walker, starting with $v'w_2.$
If again this edge is not free, then Breaker claims another arbitrary edge.

Now it is easy to see that Walker will never visit the vertex $w_2$.
In particular, she never claims the edge $w_1w_2$, which guarantees $d_W(w_1)\leq 1$ and $d_W(w_2)= 0$
throughout the game. That is, both vertices $w_1$ and $w_2$ cannot belong to a cycle of Walker,
and thus Breaker prevents cycles of length larger than $n-2$.
 
So, from now on let us focus on {\bf Walker's part}, and let us assume that Breaker is the first player.
We start with the following useful lemma,
which roughly says that in case Walker manages to create a large cycle that touches almost every vertex, while
certain properties on the distribution of Breaker's edges hold, Walker has a strategy to create
an even larger cycle within a small number of rounds. 

\begin{lemma}\label{extend}
Let $n$ be large enough. Assume a Walker-Breaker game is in progress, where
Walker has claimed a cycle $C$ with $n-125\leq v(C)\leq n-3$,
and with her current position being a vertex $x\in V(C)$.
Assume further that Breaker has claimed at most $2n$ edges so far, and that there is at 
most one vertex $y\in V(K_n)\setminus V(C)$ with $d_B(y,V(C))\geq \frac{n}{10}$.
Then Walker has a strategy to create a cycle $C'$ with $V(C)\subsetneqq V(C')$ within less than $25$ further rounds.  
\end{lemma}

\begin{proof}
In the next rounds, Walker traverses $C$ in an arbitrary direction, i.e.~she
repeats edges that she has claimed in earlier rounds, until she reaches a vertex $v\in V(C)$ with $d_B(v)\leq \frac{n}{10}$. 
For large $n$, this will take her at most 21 rounds, as $e(B)\leq 2n+O(1)$. Once she reached such a vertex,
Walker fixes two vertices $v_1,v_2\in V(K_n)\setminus V(C)$
such that $d_B(v_i,V(C))\leq \frac{n}{10}+21$. 
It follows that
$d_B(v_1,V(C))+d_B(v_2,V(C))+d_B(v,V(C))\leq \frac{3n}{10}+42$ and so, by the pigeonhole principle, there exist three
consecutive vertices $w_1,w_2,w_3$ on the cycle $C$
such that $v\neq w_2$
and such that none of the edges between
$\{w_1,w_2,w_3\}$ and $\{v,v_1,v_2\}$ is claimed by Breaker.
Walker's next move then is to proceed from $v$ to $w_2$.
W.l.o.g. we can assume that Breaker in the following move does not claim any of the edges $w_jv_1$ with $j\in [3]$ (as otherwise, we just interchange the vertices $v_1$ and $v_2$).
Walker as next proceeds from $w_2$ to $v_1$, and in the following round she closes a cycle
of length $v(C)+1$ by proceeding to one of the 
vertices in $\{w_1,w_3\}$. 
\end{proof}

With the above lemma in hand, we now can describe a strategy for Walker to create a cycle
of length $n-2$; we show later that she can always follow it, provided $n$ is large enough.
Let $v_0$ be the start vertex of Walker. After Walker's move
in round $t$, let $U_t$
be the set of vertices not touched by Walker so far,
and let $v_t$ denote her current position. We split Walker's strategy
into the following stages.

{\bf Stage I.}
Suppose Walker's $(t+1)^{\text{st}}$ move is in Stage I.
Assume that Walker's graph is a path $P_t=(v_0,v_1,\ldots, v_t)$. 
\begin{itemize}
\item[{\bf Ia.}] If $|U_t|\leq 120$, then Walker proceeds to Stage III.
\item[{\bf Ib.}] If $|U_t|> 120$ and $v_tv_0$ is a free edge, 
then Walker takes this edge, closing a cycle, and sets $v_{t+1}:=v_0$, and $\ell:= t$. She then proceeds to Stage II. 
\item[{\bf Ic.}] If $|U_t|> 120$ and $v_tv_0$ is not free,
then Walker claims an arbitrary free edge $v_tw$ with $w\in U_t$,
and sets $v_{t+1}:=w$. She then repeats Stage I.
\end{itemize}

{\bf Stage II.} Suppose Walker's $(t+1)^{\text{st}}$ move is in Stage II. Assume that Walker's graph is a cycle $C=(v_0,v_1,\ldots, v_{\ell})$
of length $\ell + 1$,
attached to a (maybe empty) path $P_t=(v_{\ell +1},v_{\ell + 2},\ldots, v_t)$
with $v_{\ell + 1}=v_0$, and with $v_t$ being the current position
of Walker. Moreover, let $x$ denote the number of past rounds in which Walker followed Case {\bf II.c.1}.
We set
$$V_t:=\left\{v\in U_t:\ d_B(v,V\setminus U_t)\geq \frac{n}{11}\right\}.$$
Moreover, in order to keep control on the distribution of Breaker's edges
after each move of Walker in Stage II,
we say that Property $P[t+1,x,i]$ is maintained if the following inequalities hold.\\[-0.5cm]
\begin{align*}
\text{Property $P[t+1,x,i]$} :
\begin{cases}
e_B(U_{t+1})\leq 3x+4+i\\
d_B(v_{t+1},U_{t+1})+e_B(U_{t+1})=e_B(U_{t})\leq 3x + 5+i\\
e_B(\{v_1,v_{\ell}\},U_{t+1})\leq 2(3x+5+i).
\end{cases}
\end{align*}

Now, Walker considers the following cases:

\begin{itemize}
\item[{\bf IIa.}] If $|U_t|\leq 120$, then Walker proceeds to Stage IV.
\item[{\bf IIb.}] If $|U_t|>120$, and $V_t=\emptyset$, then Walker claims an edge $v_tw$ with $w\in U_t$, and sets $v_{t+1}:=w$ and 
$U_{t+1}:=U_t\setminus \{w\}$, in such a way that immediately after her move 
Property $P[t+1,x,0]$ holds. (The precise details of how to choose this edge
are given later.) Walker in the next round repeats Stage II.
\item[{\bf IIc.}] If $|U_t|>120$, and 
if $V_t\neq \emptyset$, then Walker considers two subcases:
\begin{itemize} 
\item[{\bf IIc.1.}] If there is a free edge $v_tw$ with $w\in V_t$, then Walker then claims such an edge. 
She sets $v_{t+1}:=w$ and $U_{t+1}:=U_t\setminus \{w\}$ (and thus $w\notin V_{t+1}$), and she increases $x$ by one.  
After her move, $P[t+1,x,-1]$ holds with the new value of $x$. (The precise details are given later.) Then she repeats Stage II.
\item[{\bf IIc.2.}] Otherwise, Walker proceeds to a free edge $v_tw$
with $w\in U_t$ such that $wz$ is free for every $z\in V_t$. She sets $v_{t+1}:=w$ and 
$U_{t+1}:=U_t\setminus \{w\}$. Moreover, she ensures that immediately after her move
Property $P[t+1,x,1]$ holds. (The precise details are given later.) She then repeats Stage II.
\end{itemize}

\end{itemize}

{\bf Stage III.}
Suppose Walker's $(t+1)^{\text{st}}$ move is in Stage III, and let Walker's graph be a path $(v_0,v_1,\ldots, v_t)$.
Since $|U_t|\leq 120$, we have $t\geq n-121$.
Walker then claims an arbitrary free edge $v_tv_i$
with $0\leq i\leq 4$, thus creating a cycle of length at least $n-125$. Then she proceeds
to Stage V.

{\bf Stage IV.} Let $U$ be the set of untouched vertices with $|U|\leq 120$, when Walker enters Stage~IV.
Within two rounds Walker creates a cycle of length at least $n-120$, which covers every vertex that was visited by Walker so far.
Then she proceeds to Stage V.

{\bf Stage V.} When Walker enters Stage V her graph contains a cycle
of length at least $n-125$. She finally creates a cycle of length $n-2$
by repeatedly applying the strategy given by Lemma~\ref{extend}.

It is obvious that if Walker can follow the proposed strategy, then
she will create a cycle of length $n-2$. It thus remains to convince ourselves that, for large enough $n$,
she  can indeed do so. We consider
all stages and substages separately.

{\bf Stage I.} Before discussing Stage I, let us observe the following.
\begin{observation}\label{obs:StageI}
Assume that Walker did not leave Stage I before the $(t+1)^{\text{st}}$ round. 
Then immediately after her $t^{\text{th}}$ move
her graph is a path $P_t=(v_0,v_1,\ldots, v_t)$  such that all but at most 2 Breaker
edges belong to $E(v_0,V(P_t)).$
\end{observation}
\begin{proof}
If Walker does not leave Stage I, then she always plays according to Case Ic.
Since in this stage she always proceeds from her current vertex to an untouched vertex,
it is obvious that her graph is a path. Moreover,
after round $t$, Breaker has claimed $t$ edges in total. Since Walker never followed
Case Ib in an earlier round, Breaker must have claimed $v_0v_i\in E(v_0,V(P_t))$ for every $i\in\{2,\ldots,t-1\}$.
\end{proof}

It thus follows that, whenever Walker plays according to Stage I, 
her graph is a path, i.e.~the assumption of Stage I is satisfied. 
There is nothing to prove if Walker considers Case Ia or Case Ib. Moreover, she can follow Case Ic easily,
since, by Observation \ref{obs:StageI} and by the assumption of Case Ic, before her $(t+1)^{\text{st}}$ move, 
we have $e_B(v_t,U_t)\leq 3<120\leq |U_t|$.

{\bf Stage II.} When Walker plays according to Stage II, her previous move
must have been in Stage I or Stage II. Since she only enters Stage II after closing a cycle in Case Ib, and since in Stage~II she always proceeds to a vertex from the set of untouched vertices, starting from $v_0$, it is obvious that her graph has the shape as described at the beginning of
the strategy description for Stage II.

Moreover, after a move of Walker in Stage II in round $t+1$, we have $U_t=U_{t+1}\dot\cup \{v_{t+1}\}$
and thus $d_B(v_{t+1},U_{t+1})+e_B(U_{t+1})=e_B(U_{t})$ is guaranteed immediately.

To show that Walker can always follow Stage II, one may proceed by induction on the number of rounds in that stage.

Assume first that the $(t+1)^{\text{st}}$ round is the first round in Stage II.
Then, Walker played according to Case Ib in round $t$ and in all the rounds before, she followed Case~Ic. 
In particular, $x=0$. 
Immediately before Walker's $(t+1)^{\text{st}}$ move, by Observation~\ref{obs:StageI}, 
Breaker has less than 4 edges that do not belong to $E(v_0,V(C))$, where $C$ is Walker's cycle at the end of Stage I.
In particular $V_t=\emptyset$. 
We have $|U_t|>120$, as Walker entered Stage II after Case Ib, and therefore, Walker wants to follow Case IIb. 
By our observation on the distribution of Breaker's edges,
Walker can do so, as she can easily find a vertex $w\in U_t$ such that $v_tw$ is free.
Moreover $P[t+1,x,0]$ holds then with $v_{t+1}:=w$ and $U_{t+1}=U_t\setminus\{w\}$,
as $e_B(U_{t+1})\leq e_B(U_t)\leq 4$ and $e_B(\{v_1,v_{\ell}\},U_{t+1})\leq 4$.

Assume then that the $(t+1)^{\text{st}}$ round happens in Stage II, but after the first round of Stage II,
and assume that so far Walker could follow the strategy. To show that Walker can still follow the strategy,
we discuss the different cases separately. If Walker follows Case IIa, then there is nothing to prove. 
Before discussing the other parts of Stage II, we observe the following upper bound on $x$ and the size of $V_t$.

\begin{observation}
Assume Walker plays according to Stage~II for her $(t+1)^{\text{st}}$
move, after she followed the strategy for the first $t$ rounds. Then $\max\{x,|V_t|\}\le 11$.
\end{observation}

\begin{proof}
The value of $x$ increases by one each time when Walker follows Stage IIc.1, where
she enlarges her graph by a vertex of Breaker-degree at least $\frac{n}{11}$.
If we had $x\geq 12$, then Breaker would have more than $n$ edges claimed already,
likewise if $|V_t|\geq 12$.
However, since Walker's graph contains only one cycle, we have played at most $n$ rounds, a contradiction.
\end{proof}

{\bf Case IIb.} Now, let us focus on Case IIb first and assume that so far, before this
$(t+1)^{\text{st}}$ move,
Walker could always follow the proposed strategy.
Then in round $t$, Walker played according to IIb or IIc.

Assume first that Walker played according to Case IIb in round $t$.
So, we know that before Breaker's $(t+1)^{\text{st}}$ move, Property $P[t,x,0]$ was true, where $x\leq 11$. 

If Breaker in his last move did not make any of the inequalities of Property $P[t,x,0]$  invalid,
then Walker takes $w\in U_t$ arbitrarily with $v_tw$ being free.
This is possible, as \linebreak $d_B(v_t,U_t)\leq 3x+5\leq 38<|U_t|$, and it also guarantees Property $P[t+1,x,0]$
immediately after Walker's move.
Otherwise, Breaker makes at least one inequality of Property $P[t,x,0]$ invalid. There are three cases to consider,
which we discuss in the following.

{\bf Case 1.}
If Breaker with his $(t+1)^{\text{st}}$ move achieved that $e_B(U_t)=3x+5$, then in this move he claimed an edge in $U_t$.
So, we then obtain that
$d_B(v_{t},U_{t})+e_B(U_{t})\leq 3x + 6$ and thus $d_B(v_{t},U_{t})\leq 1$,
and that $e_B(\{v_1,v_{\ell}\},U_{t})\leq 2(3x+5).$
Now, Walker finds a vertex $w\in U_t$ with $v_tw$ being free such that
$d_B(w,U_t)\geq 1$. Walker claims such an edge, setting $v_{t+1}:=w$,
and then $P[t+1,x,0]$ holds, since
$e_B(U_{t+1})=e_B(U_t)-d_B(w,U_t)\leq 3x+4$, \linebreak
$d_B(v_{t+1},U_{t+1})+e_B(U_{t+1})\leq 3x + 5$,
and $e_B(\{v_1,v_{\ell}\},U_{t+1})\leq e_B(\{v_1,v_{\ell}\},U_{t})\leq 2(3x+5)$.

{\bf Case 2.}
If after Breaker's $(t+1)^{\text{st}}$ move $e_B(U_t)\leq 3x+4$ still holds,
but we have  \linebreak $d_B(v_{t},U_{t})+e_B(U_{t})= 3x + 6$, then we know that
Breaker claimed an edge in $U_t\cup \{v_t\}$. Moreover, \linebreak
$d_B(v_{t},U_{t})\leq 3x+6\leq 39<|U_t|$ and $e_B(\{v_1,v_{\ell}\},U_{t})\leq 2(3x+5).$ 
Walker then takes $w\in U_t$ arbitrarily with $v_tw$ being free. 
After Walker's move we then obtain  Property \linebreak {$P[t+1,x,0]$}, since then
$d_B(v_{t+1},U_{t+1})+e_B(U_{t+1})=e_B(U_{t})\leq 3x + 4$,
and we also have
$e_B(\{v_1,v_{\ell}\},U_{t+1})\leq e_B(\{v_1,v_{\ell}\},U_{t})\leq 2(3x+5)$.

{\bf Case 3.}
If the first two inequalities of Property $P[t,x,0]$ 
still hold after Breaker's $(t+1)^{\text{st}}$ move,
but $e_B(\{v_1,v_{\ell}\},U_{t})= 2(3x+5)+1$,
then Breaker in his move claimed an edge between $\{v_1,v_{\ell}\}$
and $U_t$.
Then there are at least
$3x+6$ vertices $w\in U_t$ with $d_B(w,\{v_1,v_{\ell}\})\geq 1$,
and for at least one such vertex $w$ Walker can claim the free edge $v_tw$,
as $d_B(v_t,U_t)\leq 3x+5$. As before, Property $P[t+1,x,0]$ is guaranteed to hold,
as then we obtain
$e_B(U_{t+1})\leq e_B(U_t)\leq 3x+4$, and
$e_B(\{v_1,v_{\ell}\},U_{t+1})=e_B(\{v_1,v_{\ell}\},U_{t})-d_B(w,\{v_1,v_{\ell}\})\leq 2(3x+5)$.

Finally, assume that Walker played according to Case IIc in the $t^{\text{th}}$ round,
and thus, before her $t^{\text{th}}$ move, we had $V_{t-1}\neq \emptyset$.
Then Walker played according to IIc.1 in round $t$, since otherwise we had $V_t\supseteq V_{t-1}\neq\emptyset$,
in contradiction to considering Case IIb for round $t+1$. In particular, the value of $x$ was increased in the $t^{\text{th}}$
round, so that $x\geq 1$, and immediately after Walker's $t^{\text{th}}$ move, we had Property $P[t,x,-1]$.
Independently of Breaker's $(t+1)^{\text{st}}$ move, Walker just chooses some vertex $w\in U_t$ with $v_tw$ being free and claims the edge $v_tw$,
which she can do since \linebreak  $d_B(v_t,U_t)\leq (3x+4)+1<|U_t|$. After proceeding as proposed, $P[t+1,x,0]$ then holds, as
$e_B(U_{t+1})\leq e_B(U_t)\leq (3x+3)+1=3x+4$, and
$e_B(\{v_1,v_{\ell}\},U_{t+1})\leq 2(3x+4)+1<2(3x+5).$

{\bf Case IIc.} Now, let us focus on Case IIc, and assume first that in round $t+1$
Walker wants to play according to Case IIc.1. He claims a free edge $v_tw$ with $w \in V_t$. By the assumption of Case IIc.1, Walker can claim the edge $v_tw$ easily. Now, let $x$ be given after the update of Stage II.c.1 in round $t+1$. Then, after Walker's move in round $t$ we had Property $P[t,x-1,1]$, independent of whether round $t$ was played in Case IIb, IIc.1 or IIc.2. Thus, no matter how Breaker chooses his $(t+1)^{\text{st}}$ edge, and no matter how 
Walker chooses $w$ above,
Property $P[t+1,x,-1]$ is maintained immediately after Walker sets $v_{t+1}=w$, $U_{t+1}=U_t\setminus\{w\}$. 
Indeed, we obtain
$e_B(U_{t+1})\leq e_B(U_t)\leq (3(x-1)+5)+1=3x+3$, and
$e_B(\{v_1,v_{\ell}\},U_{t+1})\leq e_B(\{v_1,v_{\ell}\},U_{t})\leq 2(3(x-1)+6)+1<2(3x+4).$

So, it remains to consider the case when Walker plays according to Case IIc.2 in round $t+1$. Then, after Walker's move in round $t$ 
we had Property $P[t,x,i]$, with $i\in \{-1,0,1\}$
depending on whether round $t$ was played in Case IIc.1, IIb or IIc.2, respectively. 
In any case this gives Property $P[t,x,1]$.
Using $\max\{x, |V_t|\}\leq 11$, we know that after Breaker's $(t+1)^{\text{st}}$ move we
have
\begin{align*}
\sum_{v\in V_t} d_B(v,U_t)+d_B(v_t,U_t) & \leq 2e_B(U_t)+d_B(v_t,U_t) \\ & \leq 
2(e_B(U_t)+d_B(v_t,U_t))\leq 2((3x+6)+1)\leq 80 <|U_t\setminus V_t|.
\end{align*}
That is, Walker can choose a vertex $w$ as described in the strategy.
If Walker followed Case IIb or IIc.1 in round $t$, in which case we
even had Property $P[t,x,0]$ immediately after Walker's $t^{\text{th}}$ move,
it can be easily seen that immediately after her $(t+1)^{\text{st}}$ move, we obtain
$P[t+1,x,1]$. So, assume Walker followed Case IIc.2 in round~$t$.
Then in round $t$ Walker chose $v_t\in U_t\setminus V_t$ in such a way that $v_tz$
was free for every $z\in V_{t-1}\subseteq V_t$. However,
immediately before her move in round $t+1$ no such edge was free anymore,
since Walker again follows Case IIc.2. That is, in the current round we must have $|V_t|=1$ while Breaker in his last move claimed
the unique edge $v_tz$ with $V_t=\{z\}$.
It follows that immediately after Walker's move in round $t+1$ we have
$e_B(U_{t+1})\leq e_B(U_t)\leq 3x+5$ and 
$e_B(\{v_1,v_{\ell}\},U_{t+1})\leq 2(3x+6),$
which implies Property $P[t+1,x,1]$.

{\bf Stage III.} When Walker enters Stage III in round $t+1$, it must be that Case Ia happened.
In particular, in all rounds before entering Stage III she played according to Case Ic,
and thus, when she enters Stage III, her graph is a path
$P_t=(v_0,\ldots,v_t)$ with $n-v(P_t)=|U_t|\leq 120$,
while Breaker has claimed all the edges $v_0v_i$ with $2\leq i\leq t-1$. In particular, there has to be a free edge
$v_tv_j$ with $0\leq j\leq 4$,
and Walker thus can follow the strategy and close a large cycle,
which does not contain at most $125$ vertices.

{\bf Stage IV.} Suppose that Walker enters Stage IV in round $t+1$.
Then her graph is a cycle $C=(v_0,v_1,\ldots, v_{\ell})$
attached to a path $P_t=(v_{\ell +1},v_{\ell + 2},\ldots, v_t)$
with $v_{\ell + 1}=v_0$, and with $v_t$ being her current position.
As she played according to Stage II in the $t^{\text{th}}$
round, we know that immediately after her previous move
Property $P[t,x,1]$ was true.
In her first move in Stage~IV, Walker proceeds to a vertex $w\in U_t$
such that $wv_1$, $wv_{\ell}$ and $wv_t$ are free,
which is possible as after Breaker's $(t+1)^{\text{st}}$ move we have
$\sum_{j\in\{1,\ell,t\}}d_B(v_j,U_t)\leq 3(3x+6)+1\leq 118<|U_t|.$ 
In her second move, she then either claims $wv_1$
or $wv_{\ell}$, thus creating a cycle on $V\setminus U_{t+1}$.

{\bf Stage V.} When Walker enters Stage V, her graph contains
a cycle $C_0$ of length at least $n-125$, while less than $n$ rounds were played so far.
We further observe in the following that outside the cycle there can be at most one vertex
which has a large Breaker-degree towards the cycle. 

\begin{observation}\label{outside}
When Walker enters Stage $V$ there is at most one vertex $w\in V\setminus V(C_0)$
such that $d_B(w,V(C_0))\geq \frac{n}{11}+50$.
\end{observation}

\begin{proof}
There are two possible ways that Walker enters Stage V.

The first way is that she played according to Stage III before, which she entered
because of Case Ia. That is, Walker created a path until $n-120$ vertices were touched,
while in the meantime Breaker always blocked cycles by claiming edges that are incident with $v_0$.
It follows then that $v_0$ is the only vertex
which can have a Breaker-degree of at least $\frac{n}{11}+50$.

The second way to enter Stage V is to play according to Case IIa, when the number of untouched vertices
drops down to 120, and then to reach Stage V through Stage IV.
Assume in this case that there were two vertices $w_1,w_2\in V\setminus V(C_0)$
such that $d_B(w_i,V(C_0))\geq \frac{n}{11}+50$ for both $i\in [2]$,
when Walker enters Stage V. 
It follows then that in the last 20 rounds $t$ before entering Stage IV
both vertices were elements of the corresponding set $V_t$,
as the degree $d_B(w_i,V\setminus U_t)$  can be increased by at most by 2 in each round.
(Breaker may increase this value by one by claiming an edge incident with $w_i$,
and Walker may increase this value by decreasing the set $U_t$ of untouched vertices.)
That is, Walker always would have played according to Case IIc in all these rounds. When she played according to Case IIc.1, she walked to a vertex belonging to $V_t$, which then in Stage IV became part of her cycle.
Otherwise, when Walker played according to IIc.2, then she proceeded to a vertex $w$ such that $ww_1$ and $ww_2$ were free.
Since not both of these edges could be claimed by Breaker in the following round, Walker
followed Case IIc.1 afterwards; but as she did not proceed to $w_1$ or $w_2$,
there must have been another vertex in the current set $V_t$ considered for enlarging her path.
However, as at most $11$ vertices may reach a Breaker-degree of at least $\frac{n}{11}$, it can happen at most 9 times that Walker
chooses a vertex from $V_t$ different from $w_1$ and $w_2$ for enlarging the path.
Thus, there must have been
a round in Case IIc.1 where Walker would have chosen $w_i$ to be her next position, for some $i\in [2]$. In Stage IV this vertex $w_i$ would have become a part of Walker's cycle,
in contradiction to the assumption that $w_i\notin V(C_0)$.
\end{proof}

Now, with this observation in hand, the proof is clear.
As long as Walker does not have a cycle of length $n-2$, Walker creates
larger cycles $(C_i)_{i\ge 1}$ with $V(C_i)\subsetneqq V(C_{i+1})$, by applying Lemma \ref{extend} iteratively. As $v(C_0)\geq n-125$ at the beginning of Stage V,
and as by Lemma~\ref{extend} it takes at most 25 rounds to maintain a larger cycle,
Stage V will last less than 4000 rounds, until either Walker reaches a cycle of length $n-2$,
or Walker cannot follow her strategy anymore.
It follows, by Observation \ref{outside} and since $n$ is large,
that throughout Stage~V, 
there is always at most one vertex outside Walker's current cycle $C_i$ with Breaker-degree at least $\frac{n}{10}$
towards this cycle. Moreover, as we played at most $n$ rounds before entering Stage~IV, we also have $e(B)\leq 2n$
throughout Stage~V, for large $n$. 
Thus, throughout Stage~V the conditions of Lemma \ref{extend} are always fulfilled, and therefore
Walker can follow the proposed strategy until she reaches a cycle $C_i$ with $v(C_i)\geq n-2$. \hfill $\Box$

\subsection{The biased game}

In the following we give a proof for {\bf Theorem \ref{Walker_cycle2}}.

First of all, observe that Breaker can prevent any cycle of length larger than $n-b$. Indeed, assume that Walker starts the game,
then immediately after her first move, Breaker fixes $b$ untouched vertices $u_1,\ldots,u_b$. From that point on, he always claims
the edges between $u_1,\ldots,u_b$ and Walker's current position (in case they are still free). It follows then that Walker never visits the vertices $u_1,\ldots,u_b$ and thus,
she cannot create a cycle of length larger than $n-b$. 

Now, for Walker's part, let $0<\varepsilon<10^{-5}$. Let $b\leq \frac{n}{\ln^2 n}$, $p=\frac{\ln n\cdot \ln\ln\ln n}{n}$ and $d=2\lfloor r\rfloor + 6$,
where $r=\frac{\ln n}{\ln(\frac{n}{200b})}$. Observe that $r=O(\frac{\ln n}{\ln\ln n}).$
In the following we explain how Walker can guarantee creating a cycle of length $n-O(b)$. Walker first builds a graph $G'=(V',E')$ on $n-400b$ vertices of diameter at most $d$ within at most $6n$ rounds. She can do this according to Proposition \ref{diameter}. Immediately afterwards, look at the induced graph 
$(W\cup B)[V']$.
Since the number of edges in this graph is at most $(b+1)\cdot 6n \leq 12bn$, $V'$ contains a subset $V^*$ of size $N:=n-(400+\frac{60}{\varepsilon}) b=(1-o(1))n$ such that the induced graph 
$(W\cup B)[V^*]$ has maximum degree less than $\frac{\varepsilon n}{2}$.
If we write $F'$ for the complement of $(W\cup B)[V^*]$ over $V^*$, i.e.~$F'$ is the graph of free edges on $V^*$,
then the assumptions of Theorem \ref{random} hold. That is, $\delta(F')\geq (1-\varepsilon)N$,
as $d_{F'}(v)\geq N-\frac{\varepsilon n}{2} \geq (1-\varepsilon)N$ for every $v\in V'$, for large enough $n$.
Moreover, Walker has claimed a graph $W_0=G'$ such that between each two vertices of $V^*$ there is a path of length at most $d$. 
Now $p=\omega(\frac{\ln N}{N})$ and $\frac{\varepsilon}{30(d+1)p}=\omega(\frac{n}{\ln^2 n})>b$.
Moreover, the property $\cP=\cP(N)$ of containing a Hamilton cycle is $(p,4\varepsilon)$-resilient, as follows from
Theorem \ref{resilience_new} (applied with $N$ instead of $n$). Thus, applying Theorem
\ref{random}, we conclude that Walker has a strategy to continue the game in such a way that she creates a Hamilton cycle on $V^*$, i.e.~a cycle of length $N=n-O(b)$.
\hfill $\Box$

\subsection{Creating cycles of constant length}\label{WB:graph}

The proof of Theorem \ref{H-game} is similar to that of Theorem \ref{Walker_cycle2}.
Fix an integer $k\ge 3$. Let $d=2k+4$ and let $\gamma>0$ and $C>1$ be given
according to Theorem \ref{resilience_H}, where we may assume that $\gamma<10^{-5}$. Finally, set $b=\frac{\gamma}{1000Ck}\cdot n^{\frac{k-2}{k-1}}$.
According to Proposition \ref{diameter}, Walker in the first $6n$ rounds can create a graph $G'=(V',E')\subset K_n$ on $n-400b$ vertices with diameter at most $2\frac{\ln n}{\ln(\frac{n}{200b})}+6 \le d$. Immediately after Walker has occupied $G'$, consider the induced graph $(W\cup B)[V']$ which has
at most $6(b+1)n\leq 12bn$ edges. Then $V'$ contains a subset $V^*$ of size $N:=\frac{n}{2}$ such that the induced graph 
$(W\cup B)[V^*]$ has maximum degree less than $24b$. If we write $F'$ for the complement of $(W\cup B)[V^*]$ on the
vertex set $V^*$, then $F'$ is a graph on $N$ vertices whose minimum degree is at least $N-24b=(1-o(1))N$. 
Moreover, Walker has claimed a graph $W_0=G'$ such that between each two vertices in $V^*$ there is a path of length at most $d$.
Let $p=CN^{-\frac{k-2}{k-1}}$ and observe that
$\frac{\gamma}{120(d+1)p}> b.$
Moreover, by Theorem \ref{resilience_H}
the property $\cP=\cP(N)$ of containing a copy of $C_k$
is $(p,\gamma)$-resilient (applied with $N$ instead of $n$).
Thus, by applying Theorem \ref{random} with $\varepsilon=\frac{\gamma}{4}$, we conclude that Walker has a strategy to continue the game in such a way that she creates a copy of $C_k$ on $V^*$.
\hfill $\Box$

\section{Proof of Theorem \ref{random}}\label{WB:appendix}

Our proof follows very closely the proof of Theorem 1.5 in \cite{FKN2013}.
In particular, we will use the so called {\it MinBox game}, which was motivated
by the study of the degree game~\cite{GS2009}. The game $\hbox{\it MinBox}(n,D,\alpha,b)$ is a Maker-Breaker game 
on a family of $n$ disjoint boxes $S_1,\ldots,S_n$ with $|S_i|\geq D$ for every $i\in [n]$, where Maker claims one element and Breaker claims at most $b$ elements in each round, and where Maker wins if she manages to occupy at least $\alpha |S_i|$ elements in each box $S_i$.

Throughout such a game, by $w_M(S)$ and $w_B(S)$ we denote the number of elements that Maker and Breaker have claimed so far from the box $S$, respectively. As motivated by \cite{GS2009}, we also set $\dang(S):=w_B(S)-b \cdot w_M(S)$ for every box $S$. If not every element of a box $S$ has been claimed so far, 
then $S$ is said to be {\it free}. Moreover, $S$ is said to be {\it active} if Maker still needs to claim elements of $S$, i.e.~$w_M(S)<\alpha |S|$. The following statement holds.

\begin{theorem} [Theorem 2.3 in \cite{FKN2013}]\label{MinBox}
	Let $n,b,D\in \mathbb{N}$, let $0<\alpha<1$ be a real number, and consider the game  $\hbox{\it MinBox}(n,D,\alpha,b)$.
	Assume that Maker plays as follows: In each turn, she chooses an arbitrary free active box with maximum danger, 
	and then she claims one free element from this box. Then, proceeding according to this strategy,
	\[
	\dang(S) \leq b(\ln n +1)
	\]
	is maintained for every active box $S$ throughout the game.
\end{theorem}

We remark at this point that in \cite{FKN2013}, Breaker claims exactly $b$ elements in each round of $\hbox{\it MinBox}(n,D,\alpha,b)$. However, the condition that Breaker can claim at most $b$ elements does not change the theorem above,
as its proof in \cite{FKN2013} only uses the fact that in each round $\sum_{i\in [n]}w_B(S_i)$ can increase by at most $b$.

{\bf Proof of Theorem \ref{random}}
Let $F_0=F$ be the set of free edges, which is given by the assumption of the theorem, and let $W_0$ be the current graph of Walker.

For contradiction, let us assume that Walker does not have a strategy to occupy a graph satisfying property $\cP$.
Then we know that Breaker needs to have a strategy $S_B$ which prevents Walker from creating a graph with property $\cP$,
independent of how Walker proceeds.

In the following, we describe a randomized strategy for Walker and afterwards we show that,
playing against $S_B$, this randomized strategy lets Walker create a graph from $\cP$ with high probability,
thus achieving a contradiction. The main idea of the strategy, as motivated by \cite{FKN2013},
is as follows: throughout the game Walker generates a random graph $H\sim \gnp$ on the vertex set $V(K_n)$. Following her randomized strategy,
she then obtains that a.a.s. $d_{W\setminus W_0}(v)\geq (1-4\varepsilon)d_H(v)$ holds for every vertex $v\in V(K_n)$ by the end of the game, where $W$ again denotes Walker's graph.
Thus, by the assumption on $\cP$, we then know that $W'=W\setminus W_0$ a.a.s. satisfies property $\cP$. 

When generating the random graph $H$, Walker tosses a coin on each edge of $K_n$ independently at random (even if this edge belongs to $K_n\setminus F_0$), which succeeds with probability $p$. In case of success, Walker then declares that it is an edge of $H$, and if this edge is still free in the game on $K_n$, she claims it.

To decide which edges to toss a coin on, Walker always identifies an {\em exposure vertex} $v$ (which will be marked
with the color red ). After identification, Walker proceeds to $v$ by reusing the edges of $W_0$.
Once she has reached the vertex $v$, she tosses her coin only on edges that are incident with $v$ and for which she has not yet tossed a coin. 
When she has no success or when she has success on an edge which cannot be claimed anymore (i.e.~this edge belongs to $B\cup (K_n\setminus F_0)$), 
she declares her move a failure. If the first case happens, then we say this is a failure of type~I, and
we set $f_I(v)$ to be number of {\em failures of type I}, for which $v$ is the exposure vertex.
Otherwise, if Walker has success on an edge of $B\cup (K_n\setminus F_0)$, then it said to be
a {\em failure of type II},
and with $f_{II}(v)$ we denote the number of edges that are incident with $v$ in $K_n$, and which were failures of type II.

To reach our goal, it suffices to prove that following Walker's randomized strategy 
we a.a.s. obtain that $f_{II}(v)<3.9\varepsilon np$ holds for every vertex $v\in V(K_n)$ at the end of the game. Indeed, by a simple Chernoff-type argument \cite{AS} with $p\ge \frac{\ln n}{\varepsilon n} \ge \frac{10^5\ln n}{n}$ one can verify that a.a.s. $d_H(v)\geq 0.99np$ holds for every $v\in V$,
which then yields $f_{II}(v)<4\varepsilon d_H(v)$
and $d_{W\setminus W_0}(v)\geq (1-4\varepsilon)d_H(v)$.

As in the proof of Theorem 1.5 in \cite{FKN2013}, we also say that Walker {\it exposes} an edge $e\in E(K_n)$ 
whenever she tosses a coin for the edge $e$; and
we also consider the set $U_v \subseteq N_G(v)$ which contains those vertices $u$ for which the edge $vu$ is still not exposed. 

Now, to make sure that the failures of type II do not happen very often, we associate a box $S_v$ of size $4n$ to every vertex $v\in V(K_n)$, and we use the game
$\hbox{\it MinBox}(n,4n,\frac{p}{2},2b(d+1))$ on the family of these boxes to determine the exposure vertex.
In this game, Walker imagines playing the role of Maker. The idea behind this simulated game is
to relate Breaker's degree $d_B(v)$ to the value $w_B(S_v)$,
and to relate the number of Walker's exposure processes at $v$ to Maker's value $w_M(S_v)$.
This way, we ensure that Walker stops doing exposure processes at $v$, once $d_B(v)$ becomes large,
which helps to keep the expected number of failures of type II small.

We now come to an explicit description of Walker's (randomized) strategy. Afterwards, we show that she can follow that strategy
and that, by the end of the game, a.a.s. $f_{II}(v)<3.9\varepsilon np$ holds for every vertex $v\in V(K_n)$.

{\bf Stage I.} Suppose Walker's $t^{\text{th}}$ move is in Stage I, and let $e_1,\ldots,e_b$ be the edges that Breaker claimed
in his previous move. Moreover, let $v_{t-1}$ be Walker's current position. Then Walker at first updates the simulated game $\hbox{\it MinBox}(n,4n,\frac{p}{2},2b(d+1))$
as follows: for every vertex $u\in V$, Breaker claims $|\{i\leq b:\ u\in e_i\}|$ free elements in $S_u$. (So, in total, Breaker
receives $2b$ free elements over all boxes $S_u$.)
In the real game, she then looks for a vertex $v$ that is colored red. If such a vertex exists, then she proceeds immediately with the case disctinction below. Otherwise, if no vertex has color red, then she first does the following: she identifies a vertex $v$ for which in the simulated game $\hbox{\it MinBox}(n,4n,\frac{p}{2},2b(d+1))$, $S_v$ is a free active box of largest danger value. If no such box exists, then Walker proceeds to Stage II. Otherwise,
she colors the vertex $v$ red (to identify it as her exposure vertex), Maker claims an element of $S_v$
in the simulated game $\hbox{\it MinBox}(n,4n,\frac{p}{2},2b(d+1))$, and then Walker proceeds with the following 
cases:

{\bf Case 1.} $v_{t-1}\neq v$. Let $P=(v_{t-1},x_1,\ldots,x_r,v)$ be a shortest $v_{t-1}$-$v$-path in $W_0$.
Then Walker reuses the edge $v_{t-1}x_1$ (to get closer to $v$), makes $x_1$ her new position and finishes her move.

{\bf Case 2.} $v_{t-1} = v$, i.e.~Walker's current position is the (red) exposure vertex.
Then Walker starts her exposure process
on the edges $vw$ with $w\in U_v$. She fixes an arbitrary ordering $\sigma: [|U_v|] \rightarrow U_v$ of the vertices of $U_v$,
and she tosses her coin on the vertices of $U_v$ according to that ordering, indepentently at random, with $p$ being the probability of success.

\begin{itemize}
	\item[{\bf 2a.}] If this coin tossing brings no success, then the exposure is a failure of type $I$. So,
	Walker increases the value of $f_I(v)$ by 1. In the simulated game $\hbox{\it MinBox}(n,4n,\frac{p}{2},2b(d+1))$, Maker receives $2pn-1$
	free elements in $S_v$ (or all remaining free elements if their number is less than $2pn-1$).
	In the real game, as all edges incident with $v$ are exposed,
	$U_v$ becomes the empty set, while $v$ is removed from every other set $U_w$.
	Walker removes the color from $v$ and she finishs her move by reusing an arbitrary edge of $W_0$.
	\item[{\bf 2b.}] Otherwise, let Walker's first success happen at the $k^{\text{th}}$ coin toss. 
	We distinguish between the following two subcases.
	\begin{itemize}
		\item If the edge $v\sigma(k)$ is free, then Walker claims this edge in the real game,
		thus setting $v_t:=\sigma(k)$ for her new position. For every $i\leq k$,
		she removes $v$ from $U_{\sigma(i)}$ and $\sigma(i)$ from $U_v$; moreover she
		removes the color from $v$. In the simulated game $\hbox{\it MinBox}(n,4n,\frac{p}{2},2b(d+1))$,
		Maker claims a free element from the box $S_{\sigma(k)}$. 
		\item If the edge $v\sigma(k)$ is not free, then the exposure is a failure of type II.
		Accordingly, Walker increases the value of $f_{II}(v)$ and $f_{II}(\sigma(k))$ by 1. She updates
		the sets $U_v$ and $U_{\sigma(i)}$ as in the previous case and removes the color from $v$.
		To finish her move, she reuses an arbitrary edge of $W_0$.
	\end{itemize} 
\end{itemize}

{\bf Stage II.}  In this stage, Walker tosses her coin on every unexposed edge $uv\in E(G)$. In case of success, she declares a failure of type II for both vertices $u$ and $v$.

It is easy to see that Walker can follow the proposed strategy. Indeed, the strategy always asks her to claim an edge
which is known to be free or to belong to $W_0$
and which is incident with Walker's current position. We only need to check that Theorem \ref{MinBox} is applicable,
i.e., we need to check that in the simulated game $\hbox{\it MinBox}(n,4n,\frac{p}{2},2b(d+1))$
Breaker claims at most $2b(d+1)$ elements between two consecutive moves of Maker
in which she claims free elements from free active boxes of maximum danger. This follows from Claim
\ref{uniquered}, and the observation that Maker claims such an element when Walker colors some vertex red,
while in each round Breaker claims $2b$ elements in the simulated game.

\begin{claim}\label{uniquered}
	At any point in Stage I, at most one vertex is red. After a vertex $v$ becomes red in Stage I, it takes at most $d+1$ rounds until the color is removed and, in the following round, a new (maybe the same) vertex is colored red.
\end{claim}

\begin{proof}
	The first assertion follows from the fact that, in Stage I, Walker only colors a vertex red if there is no vertex having this color.
	
	For the second assertion, assume $v$ becomes red. Then, according to the strategy description, Walker proceeds to Stage I
	as long as $v$ is red. As long as her current position is different from $v$,
	Walker walks towards the vertex $v$ by reusing the edges of $W_0$. By the assumption on $W_0$
	we know that this takes at most $d$ rounds. Once $v$ is Walker's position, the exposure process starts (which lasts only one round) and independent of its outcome, Walker removes the color from $v$.
\end{proof}

Thus, it remains to prove that, by the end of the game, a.a.s. $f_{II}(v)\leq 3.9np$ holds for every vertex $v\in V(K_n)$.
To do so, we verify the following claims. 

\begin{claim}\label{W1}
	During Stage I, $w_B(S_v) < n$ and $w_M(S_v) < (1+2p)n$ for every $v\in V$.
\end{claim}

\begin{proof}
	According to the strategy description, Breaker claims an element of $S_v$
	in the simulated game if and only if in the real game he claims
	an edge incident with $v$. Thus, $w_B(S_v)<n$ follows.
	Moreover, we observe the following:
	$w_M(S_v)$ is increased by 1 each time $v$ is colored red,
	and it is increased by at most 1 when Walker has success on an edge $vw$ where $w$
	is the red vertex (Case 2b). Both cases together can happen at most $n-1$ times,
	since we can have at most $n-1$ exposure processes in which we toss a coin on an edge
	that is incident with $v$.
	Additionally, $w_M(S_v)$ increases by at most $2pn-1$ when $v$ is the exposure vertex
	and a failure of type I happens (Case 2a). However, this can happen at most once, since after
	a failure of type I (Case 2a), Walker ensures that in the simulated game $S_v$ is not
	free or active anymore and thus, $v$ will not become the exposure vertex again.  
	Thus, the bound on $w_M(S_v)$ follows.
\end{proof}

\begin{claim}\label{W2}
	For every vertex $v\in V(G)$, $S_v$ becomes inactive before $d_B(v)\geq \frac{\varepsilon n}{5}$.	
\end{claim}

\begin{proof}
	Assume to the contrary that $w_B(S_v)=d_B(v) \geq \frac{\varepsilon n}{5}$ for some active box $S_v$. Then, by Theorem \ref{MinBox}, 
	$w_B(S_v)-2(d+1)b\cdot w_M(S_v) \leq 2(d+1)b(\ln n +1)$. 
	With $b=\frac{\varepsilon}{30(d+1)p}$ we then conclude
	$w_M(S_v)\geq 3pn-(\ln n + 1) >2pn$, where in the second inequality we used the fact that $p\ge \frac{\ln n}{\varepsilon n} \ge \frac{10^5\ln n}{n}$. 
	However, this contradicts with the assumption that $S_v$ is active. 
\end{proof}

\begin{claim}\label{W3}
	A.a.s. for every vertex $v\in V$ the following holds: As long as $U_v\neq \emptyset$ holds, we have that $S_v$ is active.
	In particular, a.a.s. every edge of $K_n$ will be exposed in Stage I.
\end{claim}

\begin{proof}
	Suppose there is a vertex, say $v$, with $U_v\neq \emptyset$ such that $S_v$ is not active. Then $f_{I}(v)=0$ and
	$2np=\frac{p}{2}|S_v| \leq w_M(S_v)$. As discussed in the previous proof, $w_M(S_v)$ could always increase by 1
	when Walker had success on an edge $vw$ where $w$ was the exposure vertex  (Case 2b), or 
	when $v$ was colored red. Notice that in the second case, Walker then exposed edges at $v$ and (besides maybe the last exposure
	process at $v$) she had success on some edge, as $f_I(v)=0$. But this means that Walker had success on at least $2pn-1$
	edges incident with $v$, i.e., $d_H(v)\geq 2np-1$. However, a simple Chernoff argument \cite{AS} shows that for $H\sim\gnp$
	a.a.s. for all vertices $v$ the last inequality will not happen. Thus, the first statement follows.
	Now, let us condition on the first statement and assume that there is an edge $uv$ of $K_n$ which is not exposed at the
	end of Stage I. Then $U_v\neq\emptyset$ and therefore $S_v$ is active.
	Moreover, by Claim~\ref{W1}, $S_v$ is free, as $w_M(S_v)+w_B(S_v)<|S_v|$. But this is in contradiction with
	the fact that Walker does not continue with Stage I.
\end{proof}

\begin{claim}\label{W4}
	A.a.s. for every vertex $v\in V(G)$, we have $f_{II}(v)\leq 3.9 \varepsilon np$.
\end{claim}

\begin{proof}
	We may condition on the statements that a.a.s. hold according to Claim \ref{W3}.
	In particular, all failures of type II happen in Stage I.
	Moreover, by Claim \ref{W2} we then have $d_B(v)\leq \frac{\varepsilon n}{5}$ 
	as long as $U_v\neq \emptyset$, for every $v\in V$, and
	by assumption we have $d_{K_n\setminus F_0}(v)\leq \varepsilon n$.
	Now, in Stage I, a failure of type II happens only if Walker has success on an edge $e$
	which already belongs to Breaker's graph, i.e.~$e\in E(B)$, or which is already claimed at the beginning,
	i.e.~$e\in E(K_n\setminus F_0)$. 
	In particular, for every $v\in V$ there is a non-negative integer $m\leq 1.2\varepsilon n$
	such that $f_{II}(v)$ is dominated by $Bin(m,p)$. Applying a Chernoff argument \cite{AS} with $p\ge \frac{\ln n}{\varepsilon n} \ge \frac{10^5 \ln n}{n}$, we obtain that a.a.s. $f_{II}(v)\leq 3.9\varepsilon np$ for every vertex $v\in V$.
\end{proof}

The last claim completes the proof of Theorem \ref{random}.\hfill $\Box$

\section{Concluding remarks}\label{WB:remarks}

{\bf Creating any cycle.} The following natural question springs to mind.
\begin{problem}
What is the largest bias $b$ for which Walker has a strategy to create a cycle (of any length)
in the $(1:b)$ Walker-Breaker game on $K_n$?
\end{problem}
All we know about this game is that the bias $b$ must be smaller than $\lceil \frac{n}{2}\rceil-1$. This is straightforward from a result of Bednarska and Pikhurko \cite{BP2005} which tells us in the biased $(1:b)$ Maker-Breaker on $K_n$, Maker can build a cycle (no matter who starts) if and only if $b<\lceil \frac{n}{2}\rceil-1$.

{\bf Creating large subgraphs.}
In Theorem \ref{H-game} we studied (biased) Walker-Breaker games in which Walker aims to
create a copy of some fixed graph of constant size. It also seems to be interesting to study
which large subgraphs Walker can create in the unbiased game on $K_n$. 
As Breaker can prevent Walker from visiting every vertex of $K_n$,
Walker cannot hope to occupy spanning structures.
As shown in Theorem \ref{Walker_cycle1}, Walker however can occupy a cycle of length $n-2$,
and applying a similar method as in the proof of Theorem \ref{Walker_cycle1},
we are also able to show that Walker can create a path of length $n-2$ (i.e.~with $n-1$ vertices)
within $n$ rounds. We wonder which other graphs (e.g.~trees) on $n-1$ vertices Walker can create.
Moreover, as already asked by Espig, Frieze, Krivelevich and Pegden~\cite{EFKP2014}, it seems to be challenging to find
the size of the largest clique that Walker can occupy. Notice that it is not hard to see that the answer
is of order $\Theta(\ln n)$.

\begin{problem}
Determine the largest $k$ such that Walker has a strategy to create a clique of size $k$.
\end{problem}

{\bf Visiting as many edges as possible.}
For constant bias $b$, we are able to show that Walker can visit
$\frac{1}{b+1}\binom{n}{2}-\Theta(n)$ edges.
However, our proof idea
heavily makes use of the possibility of repeating edges. 
To make Walker's life harder, it seems natural to study a variant where she is not allowed to
choose edges twice. We wonder how many edges Walker can occupy
in the unbiased Walker-Breaker game on $K_n$, under this restriction. One easily proves
that the answer to this question is of order $\Theta(n^2)$.

\begin{problem}
How many edges can Walker claim in the unbiased Walker-Breaker game
when she is not allowed to walk along eges twice?
\end{problem}

{\bf Doubly biased games.}
As observed earlier, Walker cannot hope to create spanning structures, when her bias equals 1,
as Breaker can isolate vertices easily in this case. This situation changes immediately when we
increase Walker's bias. In particular, the following problems seem to be of particular interest.

\begin{problem}
What is the largest bias $b$ for which Walker has a strategy to create a spanning tree of $K_n$
in the $(2:b)$ Walker-Breaker game on $K_n$?
\end{problem}

\begin{problem}
Is there a constant $c>0$ such that Walker has a strategy
to occupy a Hamilton cycle of $K_n$
in the $(2:\frac{cn}{\ln n})$ Walker-Breaker game on $K_n$?
\end{problem}

\section*{Acknowledgement}
The first author was supported by DFG project SZ 261/1-1. The research of the second author was supported by DFG within the Research Training Group ``Methods for Discrete Structures''.

\end{document}